\documentclass[10pt]{amsart}

\newtheorem{lem}{Lemma}[section]
\newtheorem{thm}{Theorem}[section]
\newtheorem{prop}{Proposition}[section]
\newtheorem{cor}{Corollary}[section]
\newtheorem{obs}{Remark}[section]
\newtheorem{defin}{Definition}[section]

\usepackage{amssymb}
\usepackage{amsmath}
\usepackage{amsfonts}
\usepackage{graphicx}
\usepackage{amsfonts,psfrag}

\numberwithin{equation}{section}

\def\k0{\kappa_0}

\def\lgl{\langle}
\def\rgl{\rangle}

\def\bfu{{\bf{u}}}

\def\bfx{{\bf{x}}}
\def\bfy{{\bf{y}}}
\def\bfo{{\bf{0}}}

\def\bfv{{\bf{v}}}

\def\mR{{\mathbb{R}^2}}
\def\bpsi{{\bar{\psi}}}
\def\bphi{{\bar{\phi}}}
\def\be{{\bar{e}}}
\def\bE{{\bar{E}}}
\def\bPhi{{\bar{\Phi}}}
\def\tpsi{{\tilde{\psi}}}
\def\tphi{{\tilde{\phi}}}
\def\te{{\tilde{e}}}
\def\tE{{\tilde{E}}}
\def\tPhi{{\tilde{\Phi}}}
\def\ttau{{\tilde{\tau}}}
\def\ubx{{\bar{\bf y}}}
\def\lbx{{\underline{\bf y}}}

\begin{document}
\title[2D turbulence in NSE]
{2D turbulence in physical scales of the Navier-Stokes equations }
\author{R. Dascaliuc}
\address{Department of Mathematics\\
University of Virginia\\ Charlottesville, VA 22904}
\author{Z. Gruji\'c}
\address{Department of Mathematics\\
University of Virginia\\ Charlottesville, VA 22904}
\date{\today}
\begin{abstract}
Local analysis of the two dimensional
Navier-Stokes equations is used to obtain estimates on the energy and enstrophy fluxes involving
Taylor and Kraichnan length
scales and the size of the domain. In the framework of zero driving force and non-increasing global energy,
these bounds produce sufficient conditions for existence of
the direct enstrophy and inverse energy cascades.
Several manifestations of locality of the fluxes under these
conditions are obtained. All the scales involved are {\em actual
physical scales} in $\mathbb{R}^2$ and no homogeneity
assumptions are made.
\end{abstract}
\maketitle

\section{introduction}

Following the groundbreaking ideas of Kolmogorov \cite{Kol1,Kol2,Kol3},
Batchelor, Kraichnan and Leith \cite{Bat82,Bat88,Kra67,Kra71,Lei68} established the foundations of empirical
theory of 2D turbulence (BKL theory). One of the main features of the BKL theory
 is the existence of {\em enstrophy cascade} over a
wide range of length scales, called the {\em inertial range}, where the
dissipation effects are dominated by the transport of enstrophy from
higher to lower scales. In contrast to the 3D turbulence, the energy
in 2D case is cascading toward the {\em larger scales}, a phenomenon
referred to as the {\em inverse energy cascade}. Direct enstrophy
and inverse energy cascades have been observed in physical
experiments (albeit certain difficulties exist in generating a
purely 2D turbulent flow), but theoretical justification of these
phenomena using equations of fluid motion, and in particular, the
Navier-Stokes equations (NSE), remains far from being settled.
Technical complexity of the NSE makes it difficult to establish the
conditions under which such cascades can occur. In the 2D case, the
NSE possess a number of useful regularity properties (unlike the 3D
case for which the global regularity is an open problem). However,
the dynamical complexity of the NSE makes a detailed study of their
long time behavior a difficult enterprise. Under certain conditions,
existence of the global attractors of high fractal and Hausdorff
dimensions has been established for the 2D NSE; moreover, it is
believed that these attractors become chaotic (although the proof is
elusive). For an overview of various mathematical models of
turbulence and the theory of the NSE, see, e.g.,
\cite{FMRTbook,Fbook,ES} and \cite{L-R,CFbook,Tbook1}, respectively.

Most rigorous studies of 2D NSE turbulence have been made in Fourier settings. In particular, in \cite{FJMR} the framework
of space-periodic solutions and infinite-time averages
was used to study main aspects of the BKL theory, including
establishing a sufficient condition for the enstrophy cascade.
This condition, involving Kraichnan length scale, is akin to our condition
(\ref{scales_con_fin}) obtained in section 4. In contrast to \cite{FJMR},
our goal was to work in {\em physical space} and with finite-time averages,
dealing with actual length scales in $\mR$ rather than the Fourier wave numbers.

In this paper we extend to the 2D case the ideas introduced in
\cite{DG1} to establish the existence of the energy cascade and
space locality of the flux for the 3D NSE. There, one of the
difficulties was the possible lack of regularity, which led us to
using the framework of suitable weak solutions (\cite{S,CKN}). In
2D, the difficulties lie in the need to work with higher-order
derivatives in the case of the enstrophy cascade, as well as in
dealing with a rather complex phenomenology of the 2D turbulence.

Despite these differences, the basic setting for studying energy and
enstrophy transfer in physical scales remains the same in both 3D and 2D case.
We utilize the refined cut-off functions to localize the relevant physical quantities in physical space
and then employ ensemble averages satisfying certain optimality conditions together with
{\em dynamics} of NSE to link local quantities to global ones (see \cite{DG1} for a detailed
discussion of our physical scales framework).

We restrict our study to a bounded region, a ball, in $\mR$, and consider the case
of short-time or {\em decaying turbulence} by setting the driving force to zero.
Thus, in contrast to infinite-time averages used in \cite{FJMR}, we use averages over
finite times. The time intervals considered here depend on the size
of the domain as well as the viscosity (see (\ref{T_con})). The
spatial ensemble average is taken by considering optimal  coverings of the
spatial domain with balls at various scales. Also, to exclude the situations of
the uniform growth of kinetic energy without any movement between the scales we
restrict our study to {\em physical situations} where the kinetic energy on the (global) spatial domain $\Omega$
is non-increasing, e.g., a bounded domain with no-slip boundary conditions,
or the whole space with either decay at infinity or periodic boundary conditions.

The paper is structured as follows. In section 2 we provide a brief
overview of the 2D NSE theory, noting the
relevant existence and regularity results. We also point out important differences
between 2D and 3D NSE, and how these difficulties are reflected in the differences between
2D and 3D turbulence.

Section 3 introduces the physical quantities of energy, enstrophy, and palinstrophy,
as well as energy and enstrophy fluxes adopted to our particular settings. We also define
the ensemble averages to be used throughout the paper.

The main result of section \ref{balls} is a surprisingly simple
sufficient condition for the enstrophy cascade (\ref{scales_con_fin}),
according to which the averaged enstrophy flux toward the lower scales
is nearly constant over a range of scales. This condition, involving
the Kraichnan scale and the size of the domain, is reminiscent of the
Poincar\'e inequality on a domain of the corresponding size (see
Remark \ref{Rem_4.2}). Moreover, the condition in hand would be easy
to check in physical experiments as the averages involved are very
straightforward.

Section 5 commences a study of inverse energy cascade in physical space.
The existence of such cascades in the 2D NSE solutions remains an open question.
 Several partial results exist; in particular, in the space-periodic setting the energy flux is oriented
 towards lower (Fourier) scales in the region below the scales of the body force
 (\cite{FJMR}), but existence of the cascade could not be established. In contrast,  \cite{BJF} provides a
 condition for the inverse energy cascades inside spectral gaps of the body force.
We prove that if the global Taylor scale is dominated by the linear size of the domain,
then the averaged energy flux is constant over a range of large scales and is
oriented outwards (see Theorem \ref{back_casc_thm}).

The second part of the paper concerns {\em locality} of the energy
and enstrophy fluxes. Similarly to the 3D turbulence (\cite{O}),
it is believed  that the energy and enstrophy fluxes inside the
inertial ranges of the 2D turbulent flows depends strongly on the flow in
nearby scales, the dependence on lower and much higher scales
being weak. The theoretical proof of this conjecture remained
elusive. The first quantitative results on fluxes were obtained by
early 70's (see \cite{Kr}). Much later, the authors in \cite{LF}
used the NSE in the Fourier setting to explore locality of scale
interactions for statistical averages, while the investigation in
\cite{E} revealed the locality of filtered energy flux under the assumption that the solutions to the
vanishing viscosity Euler's equations saturate a defining inequality of a suitable
Besov space (a weak scaling assumption). A more recent work
\cite{CCFS} provided a proof of the quasi-locality of the energy flux in
the Littelwood-Paley setting.

In section 6 we obtain several manifestations of the locality of both energy and enstrophy fluxes in the
physical space throughout the inertial ranges. In particular, considering dyadic
shells at the scales $2^k R$ ($k$ an integer) in the physical space, we show that both
ultraviolet and infrared locality propagate \emph{exponentially} in
the shell number $k$.

To the best of our knowledge, the condition (\ref{scales_con_fin})
is presently the only condition (in any solution setting) implying
both the existence of the inertial range and the locality of the
enstrophy flux.  The same is true for the relation (\ref{back_ene_casc_con})
which implies both inverse energy cascade and energy flux locality in
the physical scales of the 2D NSE. Finally, we point out that our approach
is valid for a variety of boundary conditions
(in particular, the no-slip, periodic, or the whole space with decay at infinity);
moreover, it does not involve any additional homogeneity assumptions
on the solutions to the NSE.


\section{preliminaries}
We consider two dimensional incompressible Navier-Stokes equations (NSE)
\begin{equation}\label{inc-nse}
\begin{aligned}
\frac{\partial}{\partial t}\bfu(t,\bfx)-\nu\Delta \bfu(t,\bfx)
+(\bfu(t,\bfx)\cdot\nabla)\bfu(t,\bfx)+\nabla p(t,\bfx)&=0\\
\nabla\cdot\bfu(t,\bfx)&=0\;,
\end{aligned}
\end{equation}
where the space variable $\bfx=(x_1,x_2)$ is in $\mathbb{R}^2$ and the time
variable $t$ is in $(0, \infty)$. The vector-valued function $\bfu=(u_1,u_2)$
and the scalar-valued function $p$ represent the fluid velocity and
the pressure, respectively, while the constant $\nu$ is the
viscosity of the fluid.

Under appropriate boundary conditions this system admits a unique solution (see \cite{Tbook1}, \cite{CFbook}),
which is analytic in both space and time. For convenience, we generally assume no-slip boundary conditions on a bounded domain
\begin{equation}\label{BC}
\left.\bfu\right|_{\partial\Omega}=0, \qquad \Omega\ \ \mbox{bounded in}\ \mR
\end{equation}
(although the results hold for the other physical boundary conditions which imply smoothness and non-increasing global energy $\int_{\Omega}|\bfu|^2$).

 Thus, if
$\phi\in\mathcal{D}((0,\infty)\times\Omega)$, $\phi\ge0$, where
 $\Omega$ be an open connected set in $\mR$, then multiplying NSE by $\phi\bfu$ and
 integrating by parts we obtain the local energy equation
\begin{equation}\label{loc_ene_ineq}
2\nu\iint|\nabla\otimes\bfu|^2\phi\,d\bfx\,dt =
\iint|\bfu|^2(\partial_t\phi+\nu\Delta\phi)\,d\bfx\,dt
+\iint(|\bfu|^2+2p)\bfu\cdot\nabla\phi\,d\bfx\,dt\;
\end{equation}
\noindent where $\mathcal{D}((0,\infty)\times\Omega)$ denotes the
space of infinitely differentiable functions with compact support in
$(0, \infty) \times \Omega$.

We also consider the vorticity form of the 2D NSE by taking the curl
of (\ref{inc-nse}) viewed as a 3D equation with the third component
zero,
\begin{equation}\label{vorticity_eq}
\frac{\partial}{\partial t}\omega -\nu\Delta\omega + (\bfu\cdot\nabla){\omega}=0,
\end{equation}
where $\omega=\nabla\times\bfu$ (with the convention $\bfu=(u_1,u_2,0)$ and $\omega=(0,0,\omega)$).

Note that for the full 3D NSE (\ref{vorticity_eq}) would contain the vortex-stretching term $(\omega\cdot\nabla)\bfu$.

Multiplying (\ref{vorticity_eq}) with $\phi\,\omega$ yields the local enstrophy equation,
\begin{equation}\label{loc_enst_eq}
2\nu\iint|\nabla\otimes\omega|^2\phi\,d\bfx\,dt =
\iint|\omega|^2(\partial_t\phi+\nu\Delta\phi)\,d\bfx\,dt
+\iint|\omega|^2\bfu\cdot\nabla\phi\,d\bfx\,dt\;.
\end{equation}

We will make the following assumptions on the domain $\Omega$ and test functions $\phi$.

First, we assume there exists $R_0$ satisfying
\begin{equation}\label{omega_ass}
R_0>0\quad\mbox{such that}\quad B(\bfo,3R_0)\subset\Omega\;
\end{equation}
where $B(\bfo,3R_0)$ represents the ball in $\mR$ centered at the
origin and with the radius $3R_0$.

Next, let $1/2\le\delta<1$. Choose $\psi_0\in\mathcal{D}(B(\bfo,2R_0))$
satisfying
\begin{equation}\label{psi0}
0\le\psi_0\le 1,\quad\psi_0=1\ \mbox{on}\ B(\bfo,R_0),
\quad\frac{|\nabla\psi_0|}{\psi_0^{\delta}}\le\frac{C_0}{R_0},
\quad\frac{|\Delta\psi_0|}{\psi_0^{2\delta-1}}\le\frac{C_0}{R_0^2}\;.
\end{equation}
 For a $T>0$ (to be chosen later), $\bfx_0\in B(\bfo,R_0)$ and $0<R\le R_0$, define
$\phi=\phi_{\bfx_0,T,R}(t,\bfx)=\eta(t)\psi(\bfx)$ to be used in (\ref{loc_ene_ineq}) and (\ref{loc_enst_eq}) where $\eta=\eta_T(t)$
and $\psi=\psi_{\bfx_0,R}(\bfx)$ are refined cut-off functions
satisfying the following conditions,
\begin{equation}\label{eta_def}
\eta\in\mathcal{D}(0,2T),\quad 0\le\eta\le1,\quad\eta=1\ \mbox{on}\
(T/4,5T/4),\quad\frac{|\eta'|}{\eta^{\delta}}\le\frac{C_0}{T}\; ;
\end{equation}
if $B(\bfx_0,R)\subset B(\bfo,R_0)$, then
$\psi\in\mathcal{D}(B(\bfx_0,2R))$ with
\begin{equation}\label{psi_def}\begin{aligned}
\quad 0\le\psi\le\psi_0,\quad\psi=1\ \mbox{on}\
B(\bfx_0,R)\cap B(\bfo,R_0),
\quad\frac{|\nabla\psi|}{\psi^{\delta}}\le\frac{C_0}{R},
\quad\frac{|\Delta\psi|}{\psi^{2\delta-1}}\le\frac{C_0}{R^2}\;,
\end{aligned}
\end{equation}
and if  $B(\bfx_0,R)\not\subset B(\bfo,R_0)$, then
$\psi\in\mathcal{D}(B(\bfo,2R_0))$ with $\psi=1\ \mbox{on}\ B(\bfx_0,R)
\cap B(\bfo,R_0)$ satisfying, in addition to (\ref{psi_def}), the
following:
\begin{equation}\label{psi_def_add1}
\begin{aligned}
&
\psi=\psi_0\ \mbox{on the part of the cone in}\ \mR\ \mbox{centered at zero and passing through}\\
& S(\bfo,R_0)\cap B(\bfx_0,R)\ \mbox{between}\  S(\bfo,R_0)\ \mbox{and}\
s(\bfo,2R_0)
\end{aligned}
\end{equation}
and
\begin{equation}\label{psi_def_add2}
\begin{aligned}
&
\psi=0\ \mbox{on}\ B(\bfo,R_0)\setminus B(\bfx_0,2R)\ \mbox{and outside the part of the cone in}\ \mR\\
 &
 \mbox{centered at zero and passing through}\ S(\bfo,R_0)\cap B(\bfx_0,2R)\\
 &
 \mbox{between}\  S(\bfo,R_0)\ \mbox{and}\ S(\bfo,2R_0).
\end{aligned}
\end{equation}

Figure \ref{ball_fig} illustrates the definition of $\psi$ in the case $B(\bfx_0,R)$ is not entirely contained in
$B(\bfo,R_0)$.

\begin{figure}
  \centerline{\includegraphics[scale=1, viewport=188 449 493 669, clip] {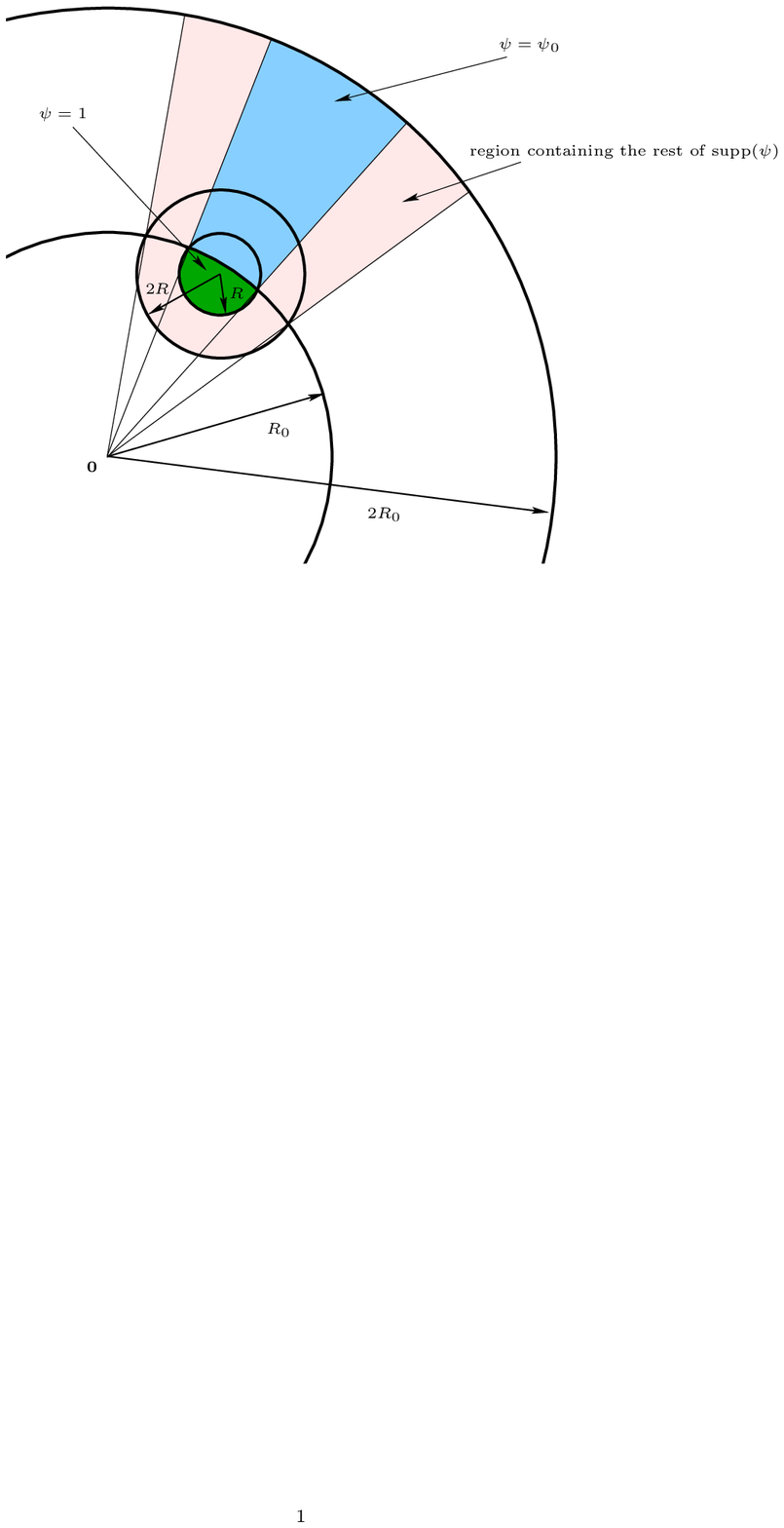}}
  \caption{Regions of supp$(\psi)$ in the case $B(\bfx_0,R)\not\subset B(\bfo,R_0)$.}
  \label{ball_fig}
\end{figure}

\begin{obs}{\em
The additional conditions on the boundary elements
(\ref{psi_def_add1}) and (\ref{psi_def_add2}) are necessary to
obtain the lower bound on the fluxes in terms of the same version of
the localized enstrophy $E$ in Theorems \ref{balls_thm} and
\ref{shells_thm} (see Remarks \ref{E'_rem1} and \ref{E'_rem3}).
}\end{obs}


\section{Averaged enstrophy and energy flux}

Let $\bfx_0\in B(\bfo,R_0)$ and $0<R\le R_0$. We define the
localized versions of energy, $e$, enstrophy, $E$, and palinstrophy,
$P$ at time $t$ associated to $B(\bfx_0,R)$ by
\begin{equation}\label{enerdef}
e_{\bfx_0,R}(t)=\int \frac{1}{2}|\bfu|^2\phi^{2\delta-1}\,d\bfx\;,
\end{equation}
\begin{equation}\label{enstdef}
E_{\bfx_0,R}(t)=\int \frac{1}{2}|\omega|^2\phi^{2\delta-1}\,d\bfx
\quad\left(\ \mbox{or}\  E'_{\bfx_0,R}(t)=\int \frac{1}{2}|\omega|^2\phi\,d\bfx
\;\right)\;,
\end{equation}
and
\begin{equation}\label{paltdef}
P_{\bfx_0,R}(t)=\int |\nabla\otimes\omega|^2\phi\,d\bfx\;.
\end{equation}

In the classical case, the total -- kinetic energy plus pressure --
flux through the sphere $S(\bfx_0,R)$ is given by
\[\int\limits_{S(\bfx_0,R)}(\frac{1}{2}|\bfu|^2+ p)\,\bfu\cdot{\bf{n}}\,ds=
\int\limits_{B(\bfx_0,R)}\left(\left(\bfu\cdot\nabla\right)\,\bfu+\nabla
p\right)\cdot\bfu\,dx\;\] where ${\bf{n}}$ is an outward normal to
the sphere $S(\bfx_0,R)$. Similarly, the enstrophy flux is given by
\[\int\limits_{S(\bfx_0,R)}\frac{1}{2}|\omega|^2\,\bfu\cdot{\bf{n}}\,ds=
\int\limits_{B(\bfx_0,R)}(\bfu\cdot\nabla){\omega}\,\cdot\omega\,dx\;.\]

Considering the NSE localized to $B(\bfx_0,R)$ leads to the
localized versions of the aforementioned fluxes,
\begin{equation}\label{en_fluxdef}
\Phi_{\bfx_0,R}(t)=\int
(\frac{1}{2}|\bfu|^2+p)\,\bfu\cdot\nabla\phi\,d\bfx\;
\end{equation}
and
\begin{equation}\label{fluxdef}
\Psi_{\bfx_0,R}(t)=\int
\frac{1}{2}|\omega|^2\,\bfu\cdot\nabla\phi\,d\bfx\;,
\end{equation}
where $\phi=\eta\psi$ with $\eta$ and $\psi$ as in
(\ref{eta_def}-\ref{psi_def}). Since $\psi$ can be constructed such
that $\nabla\phi=\eta\nabla \psi$ is oriented along the radial
directions of $B(\bfx_0,R)$ towards the center of the ball $\bfx_0$,
$\Phi_{\bfx_0,R}$ and $\Psi_{\bfx_0,R}$ can be viewed as the fluxes
{\em into} $B(\bfx_0,R)$ through the layer between the spheres
$S(\bfx_0,2R)$ and $S(\bfx_0,R)$ (in the case of the boundary
elements satisfying the additional hypotheses (\ref{psi_def_add1})
and (\ref{psi_def_add2}), $\psi$ is almost radial and the gradient
still points inward). In addition, (\ref{loc_ene_ineq}) and
(\ref{loc_enst_eq}) imply that positivity of these fluxes
contributes to the increase of $e_{\bfx_0,R}$ and $E_{\bfx_0,R}$,
respectively.

Note that the total energy flux $\Phi_{\bfx_0,R}$ consists of both
the kinetic and the pressure parts. Without imposing any specific
boundary conditions on $\Omega$ it is possible that the increase of
the kinetic energy around $\bfx_0$ is due solely to the pressure
part, without any transfer of the kinetic energy from larger scales
into $B(\bfx_0,R)$ (see \cite{DG1}). As we mentioned in the
introduction, under physical boundary conditions, like (\ref{BC}),
the increase of the kinetic energy in $B(\bfx_0,R)$ (and
consequently, the positivity of  $\Phi_{\bfx_0,R}$) implies local
transfer of the kinetic energy from larger scales simply because the
local kinetic energy is increasing while the global kinetic energy
is non-increasing resulting in decrease of the kinetic energy in the
complement. This is also consistent with the fact that in the
aforementioned scenarios one can project the NSE to the subspace of
divergence-free functions effectively eliminating the pressure and
revealing that the local flux  $\Phi_{\bfx_0,R}$ is indeed driven by
transport/inertial effects rather than the change in the pressure.

Henceforth, following the discussion in the preceding paragraph, in the setting of decaying turbulence (zero driving force, non-increasing global energy), the positivity and the negativity of  $\Phi_{\bfx_0,R}$ and  $\Psi_{\bfx_0,R}$ will be interpreted as transfer of (kinetic) energy and enstrophy around the point $\bfx_0$ at scale $R$ toward smaller scales and transfer of (kinetic) energy around the point $\bfx_0$ at scale $R$ toward larger scales, respectively.

For a quantity $\Theta=\Theta_{\bfx,R}(t)$, $t\in[0,2T]$ and a covering
$\{B(\bfx_i,R)\}_{i=1,n}$ of $B(\bfo,R_0)$ define a time-space
ensemble average
\begin{equation}
\lgl\Theta\rgl_R=\frac{1}{T}\int
\frac{1}{n}\sum\limits_{i=1}^{n}
\frac{1}{R^2}\Theta_{\bfx_i,R}(t)\,dt\;.
\end{equation}

Denote by
\begin{equation}\label{e_R_def}
e_R=\lgl e_{\bfx,R}(t)\rgl_R\;,
\end{equation}
\begin{equation}\label{E_R_def}
E_R=\lgl E_{\bfx,R}(t)\rgl_R\quad\left(\ \mbox{or}\ E'_R=\lgl E'_{\bfx,R}(t)\rgl_R\;\right)\;,
\end{equation}
\begin{equation}\label{P_R_def}
P_R=\lgl P_{\bfx,R}(t)\rgl_R\;,
\end{equation}
\begin{equation}\label{Phi_R_def}
\Phi_R=\lgl \Phi_{\bfx,R}(t)\rgl_R\;,
\end{equation}
and
\begin{equation}\label{Psi_R_def}
\Psi_R=\lgl \Psi_{\bfx,R}(t)\rgl_R\;,
\end{equation}
the averaged localized energy, enstrophy, palinstrophy, and inward-directed energy
and enstrophy fluxes over balls of radius $R$ covering $B(\bfo,R_0)$.

Also, introduce the time-space average of the localized energy, enstrophy and palinstrophy on
$B(\bfo,R_0)$,
\begin{equation}\label{e_def}
e_0=\frac{1}{T}\int
\frac{1}{R_0^2}e_{\bfo,R_0}(t)\,dt=\frac{1}{T}\frac{1}{R_0^2}\iint \frac{1}{2}|\bfu|^2\phi_0^{2\delta-1}\,d\bfx\,dt\;,
\end{equation}
\begin{equation}\label{E_def}
\begin{aligned}
&{E_0}=\frac{1}{T}\int
\frac{1}{R_0^2}E_{\bfo,R_0}(t)\,dt=\frac{1}{T}\frac{1}{R_0^2}
\iint \frac{1}{2} |\omega|^2\phi_0^{2\delta-1}\,d\bfx\,dt\;\\
&\left(\ \mbox{or}\  {E'_0}=\frac{1}{T}\int
\frac{1}{R_0^2}E_{\bfo,R_0}(t)\,dt=\frac{1}{T}\frac{1}{R_0^2}
\iint \frac{1}{2} |\omega|^2\phi_0\,d\bfx\,dt\;\right)\;,
\end{aligned}
\end{equation}
and
\begin{equation}\label{P_def}
{P_0}=\frac{1}{T}\int
\frac{1}{R_0^2}E_{\bfo,R_0}(t)\,dt=\frac{1}{T}\frac{1}{R_0^3}
\iint |\nabla\otimes\omega|^2\phi_0\,d\bfx\,dt\;
\end{equation}
where
\begin{equation}\label{phi0}
\phi_0(t,\bfx)=\eta(t)\psi_0(\bfx)
\end{equation}
with $\psi_0$ defined in (\ref{psi0}).

Finally, define Taylor and Kraichnan length scales associated with $B(\bfo,R_0)$ by
\begin{equation}\label{tau_def}
\tau_0=\left(\frac{e_0}{E'_0}\right)^{1/2}\;
\end{equation}
and
\begin{equation}\label{sigma_def}
\sigma_0=\left(\frac{E_0}{P_0}\right)^{1/2}\;.
\end{equation}

To obtain optimal estimates on the aforementioned fluxes we will
work with averages corresponding to {\em optimal} coverings of
$B(\bfo,R_0)$.

Let $K_1,K_2>1$ be absolute constants (independent of $R,R_0$, and
any of the parameters of the NSE).

\begin{defin}\label{opt_cover_def}
We say that a covering of $B(\bfo,R_0)$  by $n$ balls of radius $R$
is {\em optimal} if
\begin{equation}\label{n_con1}
\left(\frac{R_0}{R}\right)^2\le n\le K_1\left(\frac{R_0}{R}\right)^2;
\end{equation}
\begin{equation}\label{n_con2}
\mbox{any}\  \bfx\in B(\bfo,R_0)\  \mbox{is covered by at most}\ K_2 \
\mbox{balls}\ B(\bfx_i,2R)\,.
\end{equation}
\end{defin}

Note that optimal coverings exist for any $0<R\le R_0$ provided
$K_1$ and $K_2$ are large enough. In fact, the choice of $K_1$ and
$K_2$ depends only on the dimension of $\mR$, e.g, we can choose
$K_1=K_2=8$.

Henceforth, we assume that the averages $\lgl\cdot\rgl_R$ are taken
with respect to optimal coverings.

The key observation about these optimal coverings is contained in the following lemma.

\begin{lem}\label{R_ave_lem}
If the covering $\{B(\bfx_i,R)\}_{i=1,n}$ of $B(\bfo,R_0)$ is optimal then the averages
$e_R$, $E_R$, and $P_R$
satisfy
\begin{equation}\label{R_ave_est}
\begin{aligned}
&\frac{1}{K_1}e_0 \le e_R \le K_2e_0\;,\\
&\frac{1}{K_1}E_0 \le E_R \le K_2E_0
\quad\left(\; \frac{1}{K_1}E'_0 \le E'_R \le K_2E'_0\;\right)\;,\\
&\frac{1}{K_1}P_0 \le P_R \le K_2P_0\;.
\end{aligned}
\end{equation}
\end{lem}

\begin{proof}
Note that since the integrand is non-negative, using (\ref{n_con2})
and the lower bound in (\ref{n_con1}) we obtain
\[
\begin{aligned}
e_R&=\frac{1}{T}\frac{1}{R^2}
\frac{1}{n}\sum\limits_{i=1}^{n}
\iint \frac{|\bfu|^2}{2}\phi_i^{2\delta-1}\,d\bfx dt\le \frac{1}{T}\frac{1}{R^2}
\frac{1}{n} K_2\iint \frac{|\bfu|^2}{2}\phi_0^{2\delta-1}\,d\bfx dt\\
&\le
 K_2\frac{1}{T}\frac{1}{R^2}\left(\frac{R}{R_0}\right)^2\iint \frac{|\bfu|^2}{2}\phi_0^{2\delta-1}\,d\bfx dt=
 K_2e_0\;.
 \end{aligned}
\]
Next, we use the upper bound in (\ref{n_con1}) and the
non-negativity of the integrand to bound $e_R$ from below,
\[
\begin{aligned}
e_R&=\frac{1}{T}\frac{1}{R^3}
\frac{1}{n}\sum\limits_{i=1}^{n}
\iint \frac{|\bfu|^2}{2}\phi_i^{2\delta-1}\,d\bfx dt\ge\frac{1}{T}\frac{1}{R^3}
\frac{1}{n}\iint \frac{|\bfu|^2}{2}\phi_0^{2\delta-1}\,d\bfx dt\\
&\ge
\frac{1}{T}\frac{1}{R^3}\frac{1}{K_1}\left(\frac{R}{R_0}\right)^2
\iint \frac{|\bfu|^2}{2}\phi_i^{2\delta-1}\,d\bfx dt=
 \frac{1}{K_2}e_0\;,
 \end{aligned}
\]
arriving at the first relation of (\ref{R_ave_est}). The other two
relations are proved in a similar manner.

\end{proof}

Note that the lemma above shows that for the the non-negative
quantities, like energy, enstrophy, and palinstrophy, the ensemble
averages over the balls of size $R$, $e_R$, $E_R$, and $P_R$ are
comparable to the total space-time average. This is not so for the
quantities that change signs, like the energy and enstrophy fluxes.
In fact $\Phi_R$ and $\Psi_R$ provide a meaningful information as to
energy and enstrophy transfers into balls of size $R$. Positivity of
$\Psi_R$, for example, implies that there are at least some regions
of size $R$ for which the enstrophy flows inwards.

Moreover, note that the space-time ensemble  averages of energy, enstrophy, and palinstrophy
that correspond to these optimal coverings (over finite number of balls) are equivalent
to the uniform space-time average. We define the uniform space-time average of
$\Theta=\Theta_{\bfx,R}(t)$ as
\begin{equation}\label{unif_ave_def}
\Theta^u_R=\frac{1}{T}\frac{1}{R_0^2}\int\limits_{B(\bfo,R_0)}\int\limits_0^{2T} \frac{1}{R^2}\Theta_{\bfx,R}(t)\, d\bfx dt\;;
\end{equation}
thus we have the following uniform averages of energy, enstrophy,
palinstrophy and fluxes in regions of size $R$: $e^u_R$, $E^u_R$
($E'^u_R$), $P^u_R$, $\Phi^u_R$ and $\Psi^u_R$.

\begin{lem}\label{U_ave_lem}
The following estimates hold
\begin{equation}\label{unif_ave_est}
\begin{aligned}
&\frac{1}{2^2}e_0\le e^u_R \le 4^2 e_0\;,\\
&\frac{1}{2^2}E_0\le E^u_R \le 4^2 E_0
\quad\left(\; \frac{1}{2^2}E'_0\le E'^u_R \le 4^2 E'_0\;\right)\;,\\
&\frac{1}{2^2}P_0\le P^u_R \le 4^2 P_0\;.
\end{aligned}
\end{equation}
\end{lem}

\begin{proof}
We will prove the first relation in (\ref{unif_ave_est}), the others follow in a similar way.

Note that the definition of uniform average applied to the energy $e_{\bfx,R}(t)$ yeilds
\[
e^u_R = \frac{1}{R_0^2} \int\limits_{B{\bfx_0}} \left(
\frac{1}{T}\frac{1}{R^2}\iint \frac{|\bfu|^2}{2}\phi_{{\bf y},R}\,d{\bf x} dt
\right)\, d{\bf y}\;.
\]
Denote
\[F(\bfy)=\frac{1}{T}\frac{1}{R^2}\iint \frac{|\bfu|^2}{2}\phi_{{\bf y},R}\,d{\bf y} dt.\]
Observe that since the solution $\bfu$ is continuous, $F:B(\bfo,R_0)\to\mathbb{R}$ is continuous as well.

Cover $B(\bfo,R_0)$ in $n$ cubic cells, $\{C_i\}$ of linear size $R/2$. Note that
\[4\le n\le 8\]
and the area of a cell $C_i$ is
\[\mbox{A}(C_i)=\frac{R^2}{4}\;.\]
If a cell intersects the sphere $S(\bfo, R_0)$, we extend $F$ to the whole cell by setting $F(\bfy)=0$ on $C_i\setminus B(\bfo,R_0)$.
Naturally, this extension makes $F$ is measurable (but not necessarily continuous) on $\cup C_i$.

Let $\epsilon>0$. Since $F$ is bounded, there exist $\ubx_i,\lbx_i\in C_i$ such that
\[F(\ubx_i)\ge \sup\limits_{C_i}F-\frac{\epsilon}{2^i}\quad \mbox{and}\quad
F(\lbx_i)\le \inf\limits_{C_i}F+\frac{\epsilon}{2^i}\;.
\]

Consequently,
\[\begin{aligned}
\frac{1}{R_0^2}\int\limits_{B(\bfo,R_0)} F(\bfy)\, d\bfy & =\frac{1}{R_0^2}\int\limits_{\cup C_i} F(\bfy)\, d\bfy
\le\frac{1}{R_0^2}\sum\limits_{i=1}^{n}\left(F(\ubx_i)+\frac{\epsilon}{2^i}\right)A(C_i)\\ &\le
\frac{1}{4}\left(\frac{R}{R_0}\right)^2\sum\limits_{i=1}^{n}F(\ubx_i)+\frac{1}{4}\left(\frac{R}{R_0}\right)^2\epsilon\;.
\end{aligned}
\]
Note that since $F\ge 0$ and $F=0$ outside $B(\bfo,R_0)$, without loss of generality we may assume $\ubx_i\in B(\bfo,R_0)$.
Moreover, the balls $\{B(\ubx_i,R)\}$ form an optimal covering of $B(\bfo, R_0)$ in the sense of Definition \ref{opt_cover_def} with
$K_2=8^2$. Thus, 
\[
\sum\limits_{i=1}^{n} R^2 F(\ubx_i)=\sum\limits_{i=1}^{n}\frac{1}{T}\iint \frac{|\bfu|^2}{2}\phi_{\ubx_i,R}\, d\bfx dt
\le K_2 \frac{1}{T}\iint \frac{|\bfu|^2}{2}\phi_0\, d\bfx dt = K_2R_0^2 e_0\;,
\]
and so
\[
e^u_R=\frac{1}{R_0}^2 \int\limits_{B(\bfo,R_0)} F(\bfy)\, d\bfy\le \frac{K_2}{4}e_0+\frac{1}{4}\left(\frac{R}{R_0}\right)^2\epsilon\;,
\]
for any $\epsilon>0$, which implies the upper bound in the first relation in (\ref{unif_ave_est}).

To obtain the lower bound, proceed similarly,
\[\begin{aligned}
\frac{1}{R_0^2}\int\limits_{B(\bfo,R_0)} F(\bfy)\, d\bfy & =\frac{1}{R_0^2}\int\limits_{\cup C_i} F(\bfy)\, d\bfy
\ge\frac{1}{R_0^2}\sum\limits_{i=1}^{n}\left(F(\lbx_i)-\frac{\epsilon}{2^i}\right)A(C_i)\\ &\ge
\frac{1}{4}\left(\frac{R}{R_0}\right)^2\sum\limits_{i=1}^{n}F(\lbx_i)-\frac{1}{4}\left(\frac{R}{R_0}\right)^2\epsilon\;.
\end{aligned}
\]
Note that even if $\lbx_i\not\in B(\bfo,R_0)$, we still can choose $\psi_{\lbx_i,R}$ satisfying (\ref{psi_def})-(\ref{psi_def_add2})
and so the supports of $\psi_{\lbx_i,R}$ will still cover $B(\bfo,R_0)$ and
\[
\sum\limits_{i=1}^{n} R^2 F(\lbx_i)=\sum\limits_{i=1}^{n}\frac{1}{T}\iint \frac{|\bfu|^2}{2}\phi_{\lbx_i,R}\, d\bfx dt
\ge \frac{1}{T}\iint \frac{|\bfu|^2}{2}\phi_0\, d\bfx dt = R_0^2 e_0\;.
\]

Consequently,
\[
e^u_R=\frac{1}{R_0^2} \int\limits_{B(\bfo,R_0)} F(\bfy)\, d\bfy\ge \frac{1}{4}e_0-\frac{1}{4}\left(\frac{R}{R_0}\right)^2\epsilon\;,
\]
and, since $\epsilon>0$ is arbitrary, we obtain the lower bound in the first relation of (\ref{unif_ave_est}).
\end{proof}

The lemma above allows us to to note that the estimates for the optimal ensemble
averages, $\lgl\cdot\rgl_R=\frac{1}{n}\sum_{i=1}^n\cdot\;$ that will follow
will also be valid for the uniform averages,
$\lgl\cdot\rgl_U=\frac{1}{R_0^2}\int_{B({\bf 0},R_0)}\cdot\; d\bfx$.


\section{Enstrophy  cascade}\label{balls}

Let $\{B(\bfx_i,R)\}_{i=1,n}$ be an optimal covering of $B(\bfo,R_0)$.

Note that the local enstrophy equation (\ref{loc_enst_eq}) and the
definitions of $P_R$ and $\Psi_R$
(\,see (\ref{P_R_def}) and (\ref{Psi_R_def})\,)
imply

\begin{equation}\label{ene-eq}
\Psi_R= \nu P_R -
\frac{1}{n}\sum\limits_{i=1}^{n}\frac{1}{T}\frac{1}{R^3}\iint\frac{1}{2}
|\omega|^2(\partial_t\phi_i+\nu\Delta\phi_i)\,d\bfx\,dt\;
\end{equation}
where $\phi_i=\eta\psi_i$ and $\psi_i=\psi_{\bfx_i,R}$ is the spatial cut-off on
$B(\bfx_i,2R)$ satisfying (\ref{eta_def}-\ref{psi_def_add2}).

If
\begin{equation}\label{T_con}
T\ge \frac{R_0^2}{\nu},
\end{equation}
then for any $0<R\le R_0$,

\begin{equation}\label{phi_bd}
\begin{aligned}
|(\phi_i)_t|&=|\eta_t\psi_i|\le C_0\frac{1}{T}\eta^{\delta}\psi_i\le
\nu\frac{C_0}{R^2}\phi_i^{2\delta-1}\,,\\
\nu|\Delta\phi_i|&=\nu|\eta\Delta\psi_i|\le C_0\frac{\nu}
{R^2}\eta\psi_i^{2\delta-1}\le\nu\frac{C_0}{R^2}\phi_i^{2\delta-1};
\end{aligned}
\end{equation}
hence,
\[\Psi_R\ge \nu P_R -\nu \frac{C_0}{R^2}\,E_R.\]

Using (\ref{R_ave_est}) we obtain
\begin{equation}\label{low_bd_rel}
\Psi_R\ge \nu \frac{1}{K_1}P_0 -\nu \frac{C_0K_2}{R^2}\,E_0\;
\end{equation}
leading to the following proposition.

\begin{prop}
\begin{equation}\label{low_bd}
\Psi_R\ge c_1\nu P_0\,\left(1-c_2\frac{\sigma_0^2}{R^2}\right)
\end{equation}
with $c_1=1/K_1$ and $c_2=C_0K_1K_2$ (provided conditions
(\ref{n_con1}-\ref{n_con2}) are satisfied).
\end{prop}

Suppose that
\begin{equation}\label{scales_con}
\sigma_0< \frac{\gamma}{c_2^{1/2}}R_0
\end{equation}
for some $0<\gamma<1$. Then, for any $R$,
$(c_2^{1/2}/\gamma)\,\tau_0 \le R \le R_0$,
\begin{equation}\label{lower_bd}
\Psi_R\ge{c_1}(1-\gamma^2)\nu E_0=c_{0,\gamma}\nu E_0\;
\end{equation}
where
\begin{equation}
c_{0,\gamma}={c_1}(1-\gamma^2)=\frac{1-\gamma^2}{K_1}\;.
\end{equation}

To obtain an upper bound on the averaged modified flux, note that
from (\ref{R_ave_est}),
$P_R\le{K_2}P_0$, and hence, (\ref{ene-eq}) implies
\[\Psi_R\le \nu P_R+\frac{C_0}{R^2}E_R\le\nu K_2P_0+\nu C_0K_2\frac{1}{R^2}\,E_0.\]
If the condition (\ref{scales_con}) holds for some $0<\gamma<1$,
then it follows that for any $R$, $({c_2}^{1/2}/{\gamma})\,\tau_0
\le R\le R_0$,
\begin{equation}
\Psi_R\le \nu K_2 P_0+\nu \frac{C_0K_2\gamma^2}{c_2}P_0\le
c_{1,\gamma}\nu P_0\;
\end{equation}
where
\begin{equation}
c_{1,\gamma}=K_2 \left[1+\frac{C_0\gamma^2}{c_2}\right]=K_2
\left[1+\frac{\gamma^2}{K_1K_2}\right]\;.
\end{equation}

Thus we have proved the following.

\begin{thm}\label{balls_thm}
Assume that for some $0<\gamma<1$
\begin{equation}\label{scales_con_fin}
\sigma_0< c{\gamma}\,R_0\;,
\end{equation}
where
\begin{equation}\label{c_con1}
c=\frac{1}{\sqrt{C_0K_1K_2}}\;.
\end{equation}
Then, for all $R$,
\begin{equation}\label{inert_range}
\frac{1}{c\gamma}\,\sigma_0\le R\le R_0,
\end{equation}
the averaged enstrophy flux $\Psi_R$ satisfies
\begin{equation}\label{ener_casc}
c_{0,\gamma}\nu P_0\le\Psi_R\le c_{1,\gamma} \nu P_0\;
\end{equation}
where
\begin{equation}\label{c_con2}
c_{0,\gamma}=\frac{1-\gamma^2}{K_1}\,, \quad
c_{1,\gamma}=K_2 \left[1+\frac{\gamma^2}{K_1K_2}\right]\;,
\end{equation}
and the average $\lgl\cdot\rgl_R$ is computed over a time interval
$[0,T]$ with $T\ge R_0^2/\nu$ and determined by an optimal covering
of $B(\bfo,R_0)$ (i.e., a covering satisfying (\ref{n_con1}) and
(\ref{n_con2})).
\end{thm}

\begin{obs}{\em
The theorem provides a sufficient condition for the enstrophy cascade.
If (\ref{scales_con_fin}) is satisfied, then the averaged enstrophy
flux at scales $R$, throughout the inertial range defined by
(\ref{inert_range}), is oriented inwards (i.e. towards the lower scales) and is
comparable to the average enstrophy dissipation rate in $B(\bfo,R_0)$. Note that the
averages are taken over the finite-time intervals
 $[0,T]$ with $T\ge R_0^2/\nu$ (see (\ref{T_con})\.). This lower
 bound on the length of the time interval $T$ is consistent with the
 picture of decaying turbulence; namely, small $\nu$ corresponds to
 the well-developed turbulence which then persists for a longer time
 and it makes sense to average over longer time-intervals.
}\end{obs}

\begin{obs}\label{Rem_4.2}{\em
In the language of turbulence, the condition (\ref{scales_con_fin})
simply reads that the Kraichnan \emph{micro scale} computed over the
domain in view is smaller than the \emph{integral scale} (diameter
of the domain).

On the other hand, (\ref{scales_con_fin}) is equivalent to

\[
\frac{1}{T}\iint |\omega|^2\phi_0^{2\delta-1}\,d\bfx\,dt
<\frac{\gamma^2}{C_0K_1K_2}{R_0^2} \frac{1}{T}\iint
|\nabla\otimes\omega|^2\phi_0\,d\bfx\,dt\;
\]
which can be read as a requirement that the time average of a
Poincar\'e-like inequality on $B(\bfo,2R_0)$ is not saturating; this will hold for a
variety of flows in the regions of intense fluid activity (large gradients).
}\end{obs}

\begin{obs}\label{E'_rem1}{\em
If we do not impose the additional assumptions (\ref{psi_def_add1})
and (\ref{psi_def_add2}) for the test functions on the balls
$B(\bfx_i,R)\not\subset B(\bfo,R_0)$, then the lower bounds for
$\Psi_R$ in (\ref{low_bd}) and (\ref{ener_casc}) will hold with $P$
replaced by the time-space average of  the {\em{non-localized}}
in space palinstrophy on $B(\bfo,R_0)$,
\[
P'=\frac{1}{T}\int\limits_0^{2T}
\frac{1}{R_0^2}\int\limits_{B(\bfx_\bfo,R_0)}|\nabla\otimes\omega|^2\eta\,d\bfx\,dt\;.
\]
This is the case because the estimate $P_R\ge P/K_1$ gets
replaced with
\[P_R\ge\frac{1}{K_1}P'\;.\]
}\end{obs}

\begin{obs}\label{unif_ave-obs1}{\em
If we integrate the relation (\ref{loc_enst_eq}) over $B({\bf 0}, R_0)$ (instead of
summing over the optimal covering) and use Lemma \ref{U_ave_lem}, the $\Psi_R$ in
(\ref{ener_casc})
can be replaced with the uniform averaged enstrophy flux at scales $R$,
\[\Psi^u_{R}=\frac{1}{R_0^2}\int\limits_{B({\bf 0}, R_0)} \Psi_{\bfx,R}\; d\bfx\;,\]
with $K_1=2^2$ and $K_2=4^2$.}
\end{obs}

\begin{obs}\label{ene_casc_rem}{\em
Proceeding similarly as above, but using the energy balance equation
(\ref{loc_ene_ineq}) we can derive a sufficient condition for the
\emph{forward} energy cascade; if for some $0<\gamma<1$ we have
\begin{equation}\label{ene_scales_con_fin}
\tau_0< c{\gamma}\,R_0\;,
\end{equation}
then for all $R$,
\begin{equation}\label{ene_inert_range}
\frac{1}{c\gamma}\,\tau_0\le R\le R_0,
\end{equation}
the averaged energy flux $\Phi_R$ satisfies
\begin{equation}\label{ene_casc}
c_{0,\gamma}\nu E'_0\le\Phi_R\le c_{1,\gamma} \nu E'_0\;,
\end{equation}
where the constants are the same as in Theorem \ref{balls_thm}.

Note that for a
${\bf v}$ in $H^2_0(B(\bfo,R_0))^2$
\[|\nabla\otimes{\bf v}|^2=\int |\nabla\otimes {\bf v}|^2\,d\bfx=
-\int {\bf v}\cdot \Delta {\bf v}\,d\bfx\le |{\bf v}|\,|\Delta{\bf v}|\;,\]
and thus
\[
\frac{|{\bf v}|}{|\nabla\otimes{\bf v}|}\ge\frac{|\nabla\otimes{\bf v}|}{|\Delta{\bfv}|}\;.
\]
If we extend the analogy with Poincar\'e inequalities used in Remark \ref{Rem_4.2}
to this case, then the last relation suggests that the Taylor's
length scale $\tau_0$ should dominate the Kraichnan's scale $\sigma_0$
for a variety of flows characterized by large gradients, and so the sufficient
condition for forward energy cascade,  (\ref{ene_scales_con_fin}), is potentially
more restrictive then (\ref{scales_con_fin}), which is consistent with
the arguments that in 2D flows the inertial range for (forward) energy cascade,
if exists, should be much narrower then the enstrophy inertial range.
This fact was in fact established in the Fourier settings in {\cite{D1}}.
}
\end{obs}


\section{Existence of inverse energy cascades}

Assume $\bfu$ is a solution of the NSE (\ref{inc-nse}) which satisfies no-slip boundary conditions
in some bounded region $\Omega\subset\mR$:
\begin{equation}\label{no_slip}
\left.\bfu(t,\bfx)\right|_{\partial\Omega}=0 \qquad\mbox{for all}\ t\ge0\;.
\end{equation}
For simplicity, we consider $\Omega=B(\bfo,D)$ (although more general domains would be acceptable).

Define
\begin{equation}
e=\frac{1}{T}\iint\limits_{[0,2T]\times\Omega} \frac{|\bfu|^2}{2}\eta\,d\bfx dt\;
\label{back_e_def}
\end{equation}
and
\begin{equation}\label{back_E_def}
E=\frac{1}{T}\iint\limits_{[0,2T]\times\Omega}  |\nabla\bfu|^2\eta\,d\bfx dt\;,
\end{equation}
the time-averaged energy and enstrophy in $\Omega$ (localized in time), and
\begin{equation}
\tau=\frac{e}{E}
\end{equation}
the Taylor's length-scale for $\Omega$ (here $\eta$ is a function of
time satisfying (\ref{eta_def})\,).

We assume that there exists $\gamma>0$ and a length-scale $0<R_0<D/2$ such that

\begin{equation}\label{outer_glob_assu}
e\le\gamma^2 R_0^2E \quad \mbox{or equivalently,}\quad \tau\le \gamma R_0\;.
\end{equation}

In order to define localized fluxes toward larger scales we introduce the following cut-off functions.

Let $1/2\le\delta<1$.  Define
\begin{equation}\label{D_def}
D(\bfx,R)=\Omega\setminus B(\bfx,R)\;.
\end{equation}

For an $\bfx_0$ in $\Omega$ and $R_0<R\le D/2$ define the refined cut-off functions
$\bphi=\bphi_{\bfx_0,T,R}(t,\bfx)=\eta(t)\bpsi(\bfx)$, where $\eta=\eta_T(t)$ is defined in (\ref{eta_def})
and $\bpsi=\bpsi_{\bfx_0,R}(\bfx)$ is a $C^{\infty}$ function on $\Omega$ which satisfies
\begin{equation}\label{tpsi_def}\begin{aligned}
\quad 0\le\bpsi\le1,&\quad\bpsi=1\ \mbox{on}\
D(\bfx_0,R),\quad \bpsi=0\ \mbox{on}\ B(\bfx_0,R-R_0),\\
&\mbox{with}\quad\frac{|\nabla\bpsi|}{\bpsi^{\delta}}\le\frac{C_0}{R_0}\quad
\mbox{and}\quad\frac{|\Delta\bpsi|}{\bpsi^{2\delta-1}}\le\frac{C_0}{R_0^2}\;.
\end{aligned}
\end{equation}

Figure \ref{outer_ball_fig} illustrates the definition of $\bpsi$ in the case $B(\bfx_0,R)$ is entirely contained in
$\Omega$.
\begin{figure}
     \psfrag{R}{\tiny$R$}
     \psfrag{R_0}{\tiny$R_0$}
     \psfrag{0}{\tiny$\bf{0}$}
     \psfrag{D}{\tiny$D$}
     \psfrag{x0}{\tiny$\bfx_0$}
     \psfrag{psi=0}{\tiny $\bpsi=0$}
     \psfrag{psi=1}{\tiny$\bpsi=1$}
  \centerline{\includegraphics[scale=0.8]{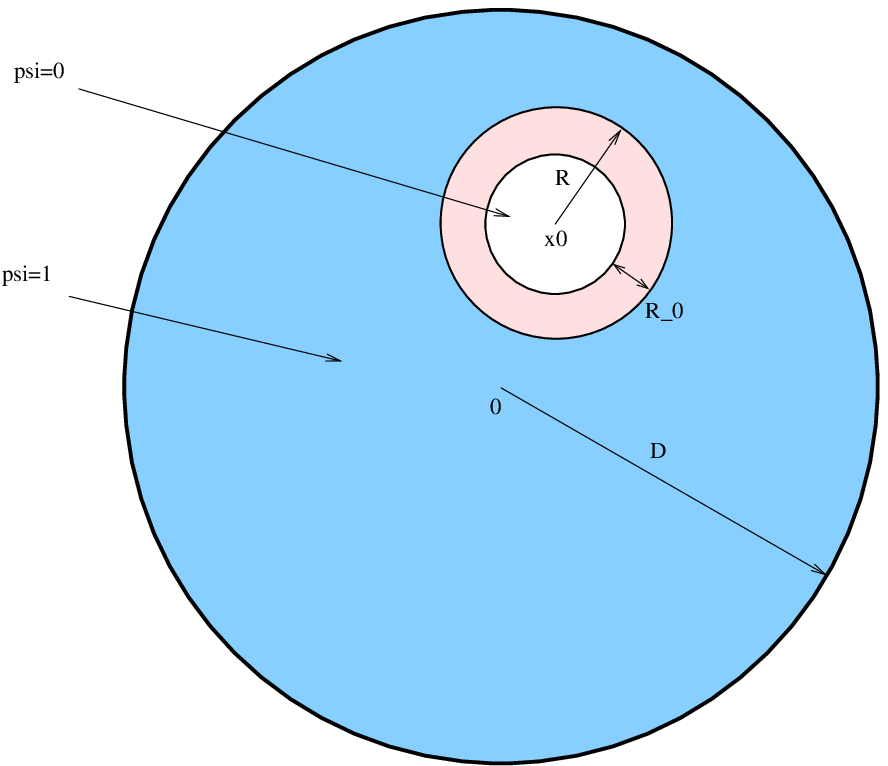}}
    \caption{Regions of support $(\bpsi_{\bfx_0,R})$.}
    \label{outer_ball_fig}
\end{figure}

Define the localized energy and enstrophy associated to the outer region $D(\bfx_0, R)$
\begin{equation}
\be_{\bfx_0,R}=\frac{1}{T}\iint \frac{|\bfu|^2}{2}\bphi_{\bfx_0,R}^{2\delta-1}\,d\bfx dt\:
\end{equation}
and
\begin{equation}
\bE_{\bfx_0,R}=\frac{1}{T}\iint |\nabla\bfu|^2\bphi_{\bfx_0,R}\,d\bfx dt\;,
\end{equation}
as well as the total energy flux
\begin{equation}
\bPhi_{\bfx_0,R}=\frac{1}{T}\iint\left(\frac{|\bfu|^2}{2}+p\right)\bfu\cdot\nabla\bphi_{\bfx_0,R}\,d\bfx dt.
\end{equation}
Note that since $\bpsi$ can be constructed such
that $\nabla\bphi=\eta\nabla \bpsi$ is oriented along the radial
directions  outside the ball $B(\bfx_0,R)$,
$\bPhi_{\bfx_0,R}$ can be viewed as the flux {\em out of} $B(\bfx_0,R)$ (i.e. {\em into} $D(\bfx_0,R)$)
through the layer between the spheres $S(\bfx_0,R)$ and
$S(\bfx_0,R-R_0)$. Additionally, (\ref{loc_ene_ineq}) confirms that $\be_{\bfx_0,R}$ tends to increase
on average in the case $\bPhi_{\bfx_0,R}>0$.

To show existence of inverse energy cascade we proceed similarly to section \ref{balls}.

Note that (\ref{no_slip}) implies that the relation (\ref{loc_ene_ineq})  holds for $\phi=\bphi$, and so, rewriting it
in terms of the quantities defined above yields
\begin{equation}\label{outer_loc_ene_bal}
\bPhi_{\bfx_0,R}=\nu\bE_{\bfx_0,R}-\frac{1}{T}\iint \frac{|\bfu|^2}{2}\left(\partial_t\bphi_{\bfx_0,R}+\nu\Delta\bphi_{\bfx_0,R}\right)\,d\bfx dt\;.
\end{equation}

Using estimates analogous to (\ref{phi_bd}) we arrive at
\begin{equation}\label{out_ene_est}
\left|\frac{1}{T}\iint \frac{|\bfu|^2}{2}\left(\partial_t\bphi_{\bfx_0,R}+\nu\Delta\bphi_{\bfx_0,R}\right)\,d\bfx dt\right|\le\frac{C_0}{R_0^2}\be_{\bfx_0,R}\;,
\end{equation}
provided
\begin{equation}\label{T_con2}
T\ge\frac{R_0^2}{\nu}\;.
\end{equation}

If $0<R_0<R<D/2$, we only need two regions $D(\bfx_1,R)$ and $D(\bfx_2,R)$ to cover $\Omega$ (by choosing $\bfx_1,\bfx_2\in\Omega$ with
$|\bfx_1-\bfx_2|>2R$). These regions provide optimal covering of $\Omega$ in the spirit of Definition \ref{opt_cover_def} which will be used in this section.

For these optimal coverings we have
\begin{equation}
\frac{1}{2}e\le \be_R=\frac{1}{2}\left(\be_{\bfx_1,R}+\be_{\bfx_2,R}\right)\le e\;
\end{equation}
and
\begin{equation}
\frac{1}{2}E\le \bE_R=\frac{1}{2}\left(\bE_{\bfx_1,R}+\bE_{\bfx_2,R}\right)\le E\;.
\end{equation}

Thus, if we sum up (\ref{outer_loc_ene_bal}) over $\bfx_1$ and $\bfx_2$ and use (\ref{out_ene_est}), we obtain the following bounds on the ensemble average of time-averaged local fluxes at scales $R$
\begin{equation}
\bPhi_R=\frac{1}{2}\left(\bPhi_{\bfx_1,R}+\bPhi_{\bfx_2,R}\right)\le \bE_R+\frac{C_0}{R_0^2}\be_R\le \nu E +\frac{C_0}{R_0^2}e\;
\end{equation}
and
\begin{equation}
\bPhi_R\ge \bE_R-\frac{C_0}{R_0^2}\be_R\ge \frac{1}{2}\nu E -\frac{C_0}{R_0^2}e\;.
\end{equation}

Consequently,

\begin{equation}\label{outer_Phi_bds}
\frac{\nu}{2} E \left(1-2C_0\frac{\tau^2}{R_0}\right)\le\bPhi_R\le \nu E \left(1+C_0\frac{\tau^2}{R_0}\right)\;.
\end{equation}

Going back to the (\ref{outer_glob_assu}) we obtain the following.

\begin{thm}\label{back_casc_thm}
Assume
\begin{equation}\label{back_ene_casc_con}
\tau\le\gamma R_0\;
\end{equation}
for some $0<R_0<D/2$ and $0<\gamma<1/\sqrt{2C_0}$.
Then, for all $R$ satisfying
\begin{equation}\label{back_iner_range}
R_0<R<\frac{D}{2}
\end{equation}
we have
\begin{equation}\label{back_ene_casc}
\bar{c}_{0,\gamma}\nu E \le\lgl\bPhi\rgl_R\le\bar{c}_{0,\gamma} \nu E\;
\end{equation}
where
\begin{equation}\label{back_casc_constants}
\bar{c}_{0,\gamma}=\frac{1}{2}(1-2C_0\gamma^2)\quad\mbox{and}\quad\bar{c}_{0,\gamma}=1+C_0\gamma^2\;,
\end{equation}
while the averages are taken with respect to optimal coverings and over time intervals $T\ge R_0^2/\nu$.

\end{thm}

\begin{obs}
{\em The meaning of the theorem above is that if the condition (\ref{back_ene_casc_con}) is satisfied, then for a range of scales $R$, the average backward energy flux is comparable to the total energy dissipation rate $\nu E$. Thus we have a backward energy cascade over the inertial range defined by (\ref{back_iner_range}). The sufficient condition  (\ref{back_ene_casc_con}) does not call for $\tau$ to be much smaller then the internal integral scale $R_0$. However, the inertial range for backwards energy cascade is wide provided $R_0\ll D/2$, which means that backwards energy cascade will exist for a wide range of scales provided $\tau\ll D$ (according to (\ref{back_iner_range}) it will start at scales comparable with $\tau$ and end at scales comparable with $D$). In particular, the scales $D$ and $R_0$ do not have to coincide.
}
\end{obs}

\begin{obs}{\em
By combining Theorem \ref{back_casc_thm} with Remark \ref{ene_casc_rem} we note that if on some ball $B(\bfx_0,R_1)\subset\Omega$ the local Taylor scale satisfies $\tau_0\ll R_1$, while the global Taylor scale $\tau\ll D$, we have both inverse energy cascade on $\Omega$ over the range of scale satisfying
(\ref{back_iner_range}) as well as the direct energy cascade inside $B(\bfx_0,R_1)$ over the range of scale defined by (\ref{ene_inert_range}) (with $R_0$ replaced by $R_1$).
}\end{obs}

\begin{obs}{\em
Since $\bfu$ is zero on $\partial\Omega$, we may replace in Theorem \ref{back_casc_thm} the ensemble average $\bPhi_R$ with the uniform space average
\[\bPhi_R^u=\frac{1}{D^2}\int\limits_{\Omega}\bPhi_{\bfx, R}\,d\bfx\;.\]
}\end{obs}

\begin{obs}{\em
We work with the no-slip boundary condition on $\Omega$, but the
results of this section (with slightly modified $e$ and $E$) will
hold for space periodic or vanishing at infinity flows as well.
}\end{obs}


\section{Locality of the averaged fluxes}\label{locality}

Let $\bfx_0\in B(\bfo,R_0)$, $0<R_2<R_1\le R_0$. In order to study the
enstrophy flux through the shell $A(\bfx_0,R_1,R_2)$ between the
spheres $S(\bfx_0,R_2)$ and $S(\bfx_0,R_1)$ we
will consider the modified cut-off functions
$\phi=\phi_{\bfx_0,T,R_1,R_2}(t,\bfx)=\eta(t)\psi(\bfx)$ to be used
in the local enstrophy balance (\ref{loc_enst_eq}) where
$\eta=\eta_T(t)$ as in (\ref{eta_def}) and
$\psi=\psi_{\bfx_0,R_1,R_2} \in \mathcal{D}(A(\bfx_0,2R_1,R_2/2))$ satisfying
\begin{equation}\label{psi_shells_def}\begin{aligned}
&0\le\psi\le\psi_0,\quad\psi=1\ \mbox{on}\
A(\bfx_0,R_1,R_2)\cap B(\bfo,R_0),\\
&\frac{|\nabla\psi|}{\psi^{\delta}}\le\frac{C_0}{\tilde{R}},\quad
\frac{|\Delta\psi|}{\psi^{2\delta-1}}\le\frac{C_0}{\tilde{R}^2}\;,
\end{aligned}
\end{equation}
where $\psi_0$ is defined in (\ref{psi0}) and
\begin{equation}\label{tilde_R}
\tilde{R}=\tilde{R}(R_1,R_2)=\min\{R_2,R_1-R_2\}\;.
\end{equation}

Use $\phi$ to define the time-averaged energy, enstrophy, and palinstrophy in the shell between
the spheres $S(\bfx_0,R_2)$ and $S(\bfx_0,R_1)$ by
\begin{equation}
\begin{aligned}
&e_{\bfx_0,R_1,R_2}=\frac{1}{T}\iint\frac{1}{2}
|\bfu|^2\phi^{2\delta-1}\,d\bfx\,dt\;,\\
&E_{\bfx_0,R_1,R_2}=\frac{1}{T}\iint\frac{1}{2}
|\omega|^2\phi^{2\delta-1}\,d\bfx\,dt
\ \left(\;  E'_{\bfx_0,R_1,R_2}=\frac{1}{T}\iint\frac{1}{2}
|\omega|^2\phi\,d\bfx\,dt\; \right)\,,\\
&P_{\bfx_0,R_1,R_2}=\frac{1}{T}\iint
|\nabla\otimes\omega|^2\phi\,d\bfx\,dt\;.
\end{aligned}
\end{equation}
Then,
\begin{equation}
\begin{aligned}
&\tau_{\bfx_0,R_1,R_2}=\left(\frac{e_{\bfx_0,R_1,R_2}}{E'_{\bfx_0,R_1,R_2}}\right)^{1/2}\;,\\
&\sigma_{\bfx_0,R_1,R_2}=\left(\frac{E_{\bfx_0,R_1,R_2}}{P_{\bfx_0,R_1,R_2}}\right)^{1/2}
\end{aligned}
\end{equation}
are the \emph{local} Taylor and Kraichnan length scales associated with the shell
$A(\bfx_0,R_1,R_2)$.

Also define the localized time-averaged flux
through the shell between the spheres $S(\bfx_0,R_2)$ and
$S(\bfx_0,R_1)$ as
\begin{equation}\label{shell_enst_flux_def}
\Psi_{\bfx_0,R_1,R_2}=\frac{1}{T}\iint
\frac{1}{2}|\omega|^2\,\bfu\cdot\nabla\phi\,d\bfx\,dt\;.
\end{equation}
Note that $\phi$ can be chosen radially (almost radially in case $A(\bfx_0,R_1,R_2)\not\subset B(\bfo,R_0)$) so that
$\Psi_{\bfx_0,R_1,R_2}=\Psi_{\bfx_0,R_1}-\Psi_{\bfx_0,R_2/2}$. Moreover, (\ref{loc_enst_eq}) implies that this flux
contributes to increase $E_{\bfx_0,R_1,R_2}$ on average. Thus $\Psi_{\bfx_0,R_1,R_2}$ can be viewed as
{\em total enstrophy flux into}  the shell $A(\bfx_0,R_1,R_2)$.

Similarly, total {\em energy} flux into the shell $A(\bfx_0,R_1,R_2)$ is defined by
\begin{equation}\label{shell_ene_flux_def}
\Phi_{\bfx_0,R_1,R_2}=\frac{1}{T}\iint
\left(\frac{1}{2}|\bfu|^2+p\right)\,\bfu\cdot\nabla\phi\,d\bfx\,dt\;.
\end{equation}

Note that $\phi$ satisfies similar estimates to (\ref{phi_bd}) (with $R$ replaced by $\tilde{R}$), and so, if $T\ge R_0^2/\nu$, the local
enstrophy balance (\ref{loc_enst_eq}) leads to

\begin{equation}
\begin{aligned}
\Psi_{\bfx_0,R_1,R_2} &\ge \nu P_{\bfx_0,R_1,R_2}-\nu \frac{C_0}{\tilde{R}^2}E_{\bfx_0,R_1,R_2}\\
&= \nu
P_{\bfx_0,R_1,R_2}\,\left(1-C_0\frac{\sigma^2_{\bfx_0,R_1,R_2}}{\tilde{R}^2}\right)\;,
\end{aligned}
\end{equation}
for any $\bfx_0\in B(\bfo,R_0)$
and any $0<R_2<R_1\le R_0$.

Similarly, utilizing (\ref{loc_enst_eq}) again, we obtain an upper
bound
\begin{equation}\label{shells_up_bd}
\begin{aligned}
\Psi_{\bfx_0,R_1,R_2} &\le\nu P_{\bfx_0,R_1,R_2}+\frac{C_0}{\tilde{R}^2}E_{\bfx_0,R_1,R_2}\\
&= \nu
P_{\bfx_0,R_1,R_2}\,\left(1+C_0\frac{\sigma^2_{\bfx_0,R_1,R_2}}{\tilde{R}^2}\right)\;,
\end{aligned}
\end{equation}

Combining the two bounds on $\Psi_{\bfx_0,R_1,R_2}$ we obtain
\begin{equation}\label{Psi_loc_bd}
 \nu P_{\bfx_0,R_1,R_2}\,\left(1-C_0\frac{\sigma^2_{\bfx_0,R_1,R_2}}{\tilde{R}^2}\right)
\le\Psi_{\bfx_0,R_1,R_2}\le
\nu P_{\bfx_0,R_1,R_2}\,\left(1+C_0\frac{\sigma^2_{\bfx_0,R_1,R_2}}{\tilde{R}^2}\right)\;;
\end{equation}
thus, we have arrived at our first locality result.

\begin{thm}\label{shells_thm1}
Let $0<\gamma<1$, $\bfx_0\in B(\bfo,R_0)$ and $0<R_2<R_1\le R_0$. If
\begin{equation}\label{scales_con2_fin}
\sigma_{\bfx_0,R_1,R_2}<\frac{\gamma}{C_0^{1/2}}\tilde{R}\;
\end{equation}
with $\tilde{R}$ defined by (\ref{tilde_R}), then
\begin{equation}\label{ener_casc1}
(1-\gamma^2)\,\nu P_{\bfx_0,R_1,R_2}\le\Psi_{\bfx_0,R_1,R_2}\le
(1+\gamma^2)\,\nu P_{\bfx_0,R_1,R_2}\;
\end{equation}
where the time average is taken over an interval of time $[0,T]$
with $T\ge R_0^2/\nu$.
\end{thm}

\begin{obs}{\em
The theorem states that if the local Kraichnan scale
$\sigma_{\bfx_0,R_1,R_2}$, associated with a shell
$A(\bfx_0,R_1,R_2)$, is smaller than the thickness of the shell
$\tilde{R}$ (a local integral scale), then the time average of the
total enstrophy flux into that shell towards its center $\bfx_0$
is comparable to the time average of the localized palinstrophy in the
shell, $P_{\bfx_0,R_1,R_2}$. Thus, under the assumption
(\ref{scales_con2_fin}) the flux through the shell
$A(\bfx_0,R_1,R_2)$ depends essentially only on the palinstrophy
contained in the neighborhood of the shell, regardless of what
happens at the other sales, making (\ref{scales_con2_fin}) a sufficient
condition for the \emph{locality} of the flux through
$A(\bfx_0,R_1,R_2)$. }\end{obs}

\begin{obs}{\em
Similarly as in the case of condition (\ref{scales_con_fin}), we can
observe that condition (\ref{scales_con2_fin}) can be viewed as a
requirement that the time average of a Poincar\'e-like inequality on
the shell is not saturating making it plausible in the case of
intense fluid activity in a neighborhood of the shell. }\end{obs}

In order to further study the locality of the enstrophy flux, we will
estimate the ensemble averages of the fluxes through the shells
$A(\bfx_i,2R,R)$ of thickness $\tilde{R}=R$. Since we are interested
in the shells inside $B(\bfo,R_0)$, we require the lattice points
$\bfx_i$ to satisfy
\begin{equation}\label{x_i_shells_con}
B(\bfx_i,R)\subset B(\bfo,R_0)\;.
\end{equation}
To each $A(\bfx_i,2R,R)$ we associate a test function
$\phi_i=\eta\psi_i$ where $\eta$ satisfies (\ref{eta_def}) and
$\psi_i$ satisfies (\ref{psi_shells_def}) with $\bfx_0=\bfx_i$ and $\tilde{R}=R$.

If  $A(\bfx_i,2R,R)\not\subset B(\bfo,R_0)$ (i.e. we have $B(\bfx_i,R)\subset B(\bfo,R_0)$
and $B(\bfx_i,2R)\setminus B(\bfo,R_0)\not=\emptyset$), then
$\psi_i\in\mathcal{D}(B(\bfo,2R_0))$ with $\psi_i=1\ \mbox{on}\ A(\bfx_0,2R,R)
\cap B(\bfo,R_0)$ satisfying, in addition to (\ref{psi_shells_def}), the
following:
\begin{equation}\label{psi_shells_def_add1}
\begin{aligned}
&
\psi_i=\psi_0\ \mbox{on the part of the cone in}\ \mR\ \mbox{centered at
zero and passing}\\ &\mbox{ through}\   S(\bfo,R_0)\cap B(\bfx_i,2R)\
\mbox{between}\  S(\bfo,R_0)\ \mbox{and}\
S(\bfo,2R_0)
\end{aligned}
\end{equation}
and
\begin{equation}\label{psi_shells_def_add2}
\begin{aligned}
&
\psi_i=0\ \mbox{on}\ B(\bfo,R_0)\setminus A(\bfx_i,4R,R/2)\ \mbox{and outside the part of the}\\ &
\mbox{cone in}\ \mR\ \mbox{centered at zero and passing through}\ S(\bfo,R_0)\cap B(\bfx_i,4R)\\
 &
 \mbox{between}\  S(\bfo,R_0)\ \mbox{and}\ S(\bfo,2R_0).
\end{aligned}
\end{equation}

Figure \ref{shall_fig} illustrates the definition of $\psi_i$ in the case
$A(\bfx_i,2R,R)$ is not entirely contained in
$B(\bfo,R_0)$.

\begin{figure}
  \centerline{\includegraphics[scale=1, viewport=173 461 509 669, clip] {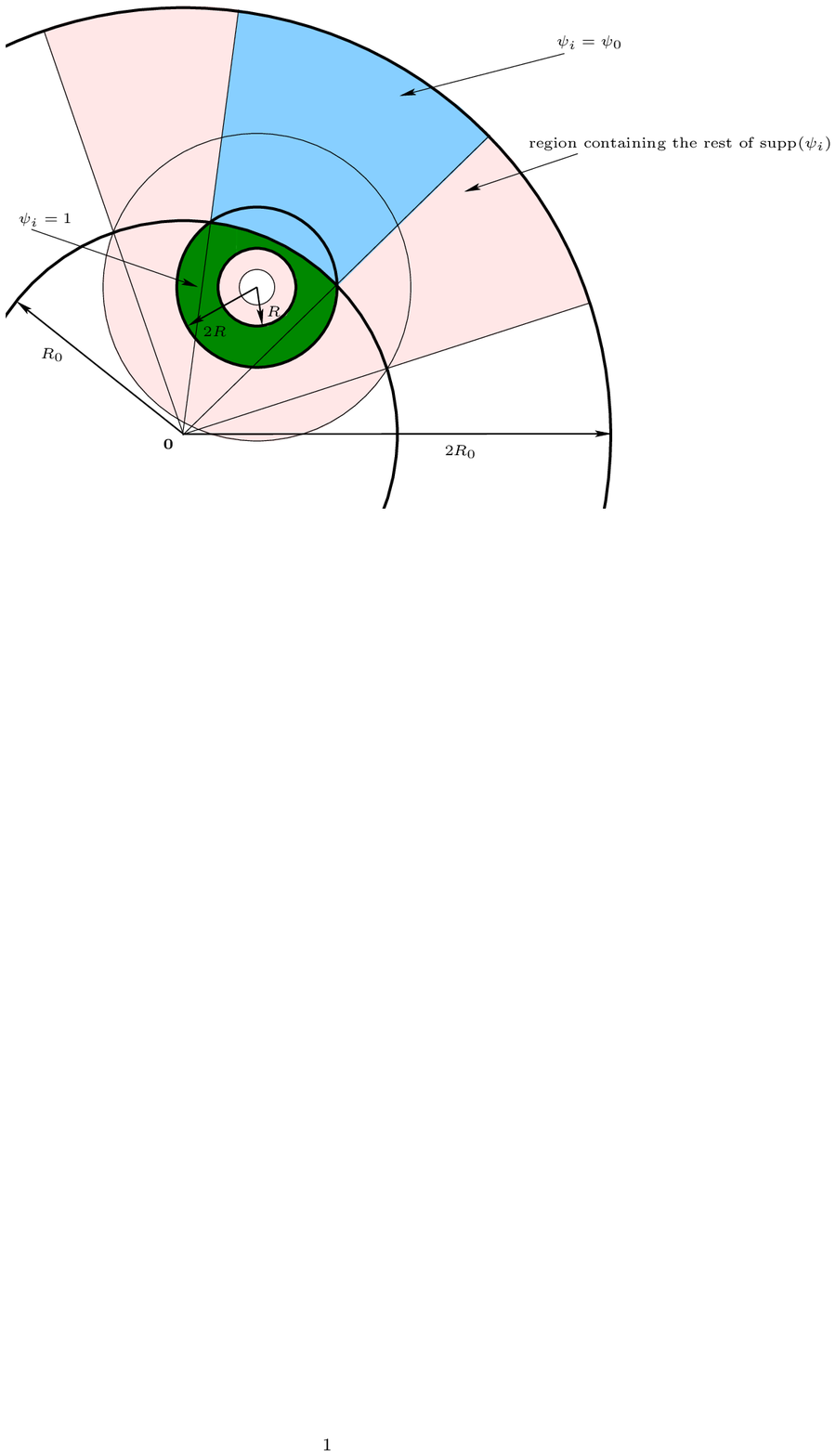}}
  \caption{Regions of supp$(\psi_i)$ in the case $A(\bfx_i,2R,R)\not\subset B(\bfo,R_0)$.}
  \label{shall_fig}
\end{figure}

Similarly as in the previous section, we consider {\em optimal}
coverings of $B(\bfo,R_0)$ by shells $\{A(\bfx_i,2R,R)\}_{i=1,n}$
such that (\ref{x_i_shells_con}) is satisfied,
\begin{equation}\label{shells_n_con1}
\left(\frac{R_0}{R}\right)^2\le n\le
K_1\left(\frac{R_0}{R}\right)^2,
\end{equation}
and
\begin{equation}\label{Shells_n_con2}
\mbox{any}\  \bfx\in B(\bfo,R_0)\  \mbox{is covered by at most}\  K_2\
\mbox{shells}\ A(\bfx_i,4R,R/2)\;.
\end{equation}

Introduce
\begin{equation}
\begin{aligned}
&\tilde{e}_{2R,R}=\frac{1}{n}\sum\limits_{i=1}^{n}e_{\bfx_i,2R,R}\;,\\
&\tilde{E}_{2R,R}=\frac{1}{n}\sum\limits_{i=1}^{n}E_{\bfx_i,2R,R}
\quad\left(\ \tilde{E}'_{2R,R}=\frac{1}{n}\sum\limits_{i=1}^{n}E_{\bfx_i,2R,R}\ \right)\;,\\
&\tilde{P}_{2R,R}=\frac{1}{n}\sum\limits_{i=1}^{n}P_{\bfx_i,2R,R}\;,
\end{aligned}
\end{equation}
and
\begin{equation}
\begin{aligned}
&\tilde{\Phi}_{2R,R}=\frac{1}{n}\sum\limits_{i=1}^{n}\Phi_{\bfx_i,2R,R}\;,\\
&\tilde{\Psi}_{2R,R}=\frac{1}{n}\sum\limits_{i=1}^{n}\Phi_{\bfx_i,2R,R}\;
\end{aligned}
\end{equation}
the ensemble averages of the time-averaged energy, enstrophy, palinstrophy, and
energy and enstrophy fluxes on the shells of thickness $R$ corresponding to the
covering $\{A(\bfx_i,2R,R)\}_{i=1,n}$\,.

Taking the ensemble averages in (\ref{loc_enst_eq}) and applying the
bounds for derivatives of $\phi_i$, we arrive at
\begin{equation}\label{low_R12_bd}
\tilde{\Psi}_{2R,R}\ge \nu\tilde{P}_{2R,R}-\nu\frac{C_0}{R^2}\,\tilde{E}_{2R,R}\;,
\end{equation}
provided $T\ge R_0^2/\nu$.

If the covering is optimal, i.e., if
(\ref{x_i_shells_con}) and (\ref{shells_n_con1}-\ref{Shells_n_con2}) hold, then
\begin{equation}\label{E_R12_E_ineq}
\tilde{P}_{2R,R}\ge \frac{1}{n}\tilde{P}\ge \frac{1}{K_1}\left(\frac{R}{R_0}\right)^{2}\tilde{P}_0
\end{equation}
and
\begin{equation}\label{e_R12_e_ineq}
\tilde{E}_{2R,R}\le \frac{K_2}{n}\tilde{E}\le
K_2\left(\frac{R}{R_0}\right)^2\tilde{E}_0\;
\end{equation}
where
\begin{equation}
\tilde{P}_0=\frac{1}{T}\iint
|\nabla\otimes\omega|^2\phi_0\,d\bfx\,dt=R_0^2\,P_0\;
\end{equation}
is the time average of the localized palinstrophy on $B(\bfo,R_0)$ and
\begin{equation}
\tilde{E}_0=\frac{1}{2}\frac{1}{T}\iint
|\omega|^2\phi_0^{2\delta-1}\,d\bfx\,dt=R_0^2\, E_0\;
\end{equation}
is the time average of  the localized enstrophy on $B(\bfo,R_0)$ with
$\phi_0$ is defined by (\ref{phi0}).

Let us note that
\begin{equation}\label{tau_obs}
\sigma_0=\left(\frac{E_0}{P_0}\right)^{1/2}=\left(\frac{\tilde{E}_0}{\tilde{P}_0}\right)^{1/2}\;.
\end{equation}

Utilizing (\ref{E_R12_E_ineq}), (\ref{e_R12_e_ineq}) and
(\ref{tau_obs}) in the inequality (\ref{low_R12_bd}) gives
\begin{equation}\label{Psi_R12_low}
\tilde{\Psi}_{2R,R}\ge
\frac{1}{K_1}\left(\frac{R}{R_0}\right)^2\,\nu\tilde{P}_0
\left(1-C_0K_1K_2\frac{\sigma_0^2}{R^2}\right)\;.
\end{equation}


Taking the ensemble averages in the localized enstrophy equation
(\ref{loc_enst_eq}) again, this time looking for an upper bound, yields
\[
\tilde{\Psi}_{2R,R}\le \nu\tilde{P}_{2R,R}+\nu\frac{C_0}{R^2}\tilde{E}_{2R,R}\;.
\]
If the covering $\{A(\bfx_i,2R,R)\}_{i=1,n}$ of $B(\bfo,R_0)$ is
optimal, then, in addition to (\ref{e_R12_e_ineq}),
\begin{equation}\label{e_R12_ineq}
\tilde{P}_{2R,R}\le \frac{K_2}{n}\tilde{P}_0\le
K_2\left(\frac{R}{R_0}\right)^2\tilde{P}_0\;;
\end{equation}
hence,
\begin{equation}
\begin{aligned}
\tilde{\Psi}_{2R,R}&\le \nu K_2\left(\frac{R}{R_0}\right)^2\tilde{P}_0
+\nu K_2\frac{C_0}{R^2}\left(\frac{R}{R_0}\right)^2\tilde{E}_0\\
&=K_2\left(\frac{R}{R_0}\right)^2\,\nu\tilde{P}_0
\left(1+C_0\frac{\sigma_0^2}{R^2}\right)\;.
\end{aligned}
\end{equation}

Collecting all the bounds on $\tilde{\Psi}_{2R,R}$ we obtain
\begin{equation}
\frac{1}{K_1}\left(\frac{R}{R_0}\right)^2\,\nu\tilde{P}_0
\left(1-C_0K_1K_2\frac{\sigma_0^2}{R^2}\right)
\le\tilde{\Psi}_{2R,R} \le
K_2\left(\frac{R}{R_0}\right)^2\,\nu\tilde{P}_0
\left(1+C_0\frac{\sigma_0^2}{R^2}\right)\;
\end{equation}
which readily implies the following theorem.

\begin{thm}\label{shells_thm}
Assume that the condition (\ref{scales_con_fin}) holds for some
$0<\gamma<1$. Then, for any $R$ satisfying (\ref{inert_range}), the
ensemble average of the time-averaged enstrophy flux into
the shells of thickness $R$, $\tilde{\Psi}_{2R,R}$,  satisfies
\begin{equation}\label{ener_casc2}
c_{0,\gamma}\left(\frac{R}{R_0}\right)^2\,\nu\tilde{P}_0\le\tilde{\Psi}_{2R,R}\le
c_{1,\gamma} \left(\frac{R}{R_0}\right)^2\,\nu\tilde{P}_0\;
\end{equation}
where $c$, $c_{0,\gamma}$, and $c_{1,\gamma}$ are defined in (\ref{c_con1})
and (\ref{c_con2})
and the average is computed over a time interval $[0,T]$ with $T\ge R_0^2/\nu$
and determined by an optimal covering $\{A(\bfx_i,2R,R)\}_{i=1,n}$ of $B(\bfo,R_0)$
(i.e. satisfying (\ref{x_i_shells_con}), (\ref{shells_n_con1}), and (\ref{Shells_n_con2})).
\end{thm}

Note that if
\[
\Psi_{2R,R}=\frac{1}{R^2}\tilde{\Psi}_{2R,R}
\]
denotes the ensemble average of the time-\emph{space} averaged
modified energy flux through the shells of thickness $R$ then,
dividing (\ref{ener_casc2}) by $R^{2}$, we obtain the following.

\begin{cor}
Under the conditions of the previous theorem,
\begin{equation}\label{ener_casc3}
c_{0,\gamma}\nu P_0\le\Psi_{2R,R}\le c_{1,\gamma} \nu P_0\;.
\end{equation}
\end{cor}

Theorem \ref{shells_thm} allows us to show locality of the time-averaged modified
enstrophy flux under the assumption (\ref{scales_con_fin}).  Indeed, the ensemble
average of the  time-averaged flux through the spheres of radius
$R$ satisfying (\ref{inert_range}) is
\[
\tilde{\Psi}_R=R^2\Psi_R\;.
\]
According to Theorem \ref{balls_thm},
\[
c_{0,\gamma}\left(\frac{R}{R_0}\right)^2\,\nu\tilde{P}_0\le\tilde{\Psi}_R\le
c_{1,\gamma} \left(\frac{R}{R_0}\right)^2\,\nu\tilde{P}_0\;.
\]
On the other hand, the ensemble average of the flux through the shells between spheres
of radii $R_2$ and $2R_2$, according to Theorem \ref{shells_thm} is
\[
c_{0,\gamma}\left(\frac{R_2}{R_0}\right)^2\,\nu\tilde{P}_0\le\tilde{\Psi}_{2R_2,R_2}\le
c_{1,\gamma} \left(\frac{R_2}{R_0}\right)^2\,\nu\tilde{P}_0\;.
\]
Consequently,
\begin{equation}\label{time_locality}
\frac{c_{0,\gamma}}{c_{1,\gamma}}\left(\frac{R_2}{R}\right)^2\le
\frac{\tilde{\Psi}_{2R_2,R_2}}{\tilde{\Psi}_R}\le\frac{c_{1,\gamma}}{c_{0,\gamma}}
\left(\frac{R_2}{R}\right)^2\;.
\end{equation}

Thus, under the assumption (\ref{scales_con_fin}), throughout the
inertial range given by (\ref{inert_range}), the contribution of the
shells at scales comparable to $R$ is  comparable to the total flux
at scales $R$, the contribution of the the shells at scales $R_2$
much smaller than $R$ becomes negligible (ultraviolet locality) and
the flux through the shells at scales $R_2$ much bigger than $R$
becomes substantially bigger and thus essentially uncorrelated to
the flux at scales $R$ (infrared locality).

Moreover, if we choose $R_2=2^kR$ with $k$ an integer, the relation (\ref{time_locality})
becomes
\begin{equation}\label{exp_time_locality}
\frac{c_{0,\gamma}}{c_{1,\gamma}}2^{2k}\le
\frac{\tilde{\Psi}_{2^{k+1}R,2^k{R}}}{\tilde{\Psi}_R}\le\frac{c_{1,\gamma}}{c_{0,\gamma}}
2^{2k}\;,
\end{equation}
which implies that the aforementioned manifestations of locality propagate
\emph{exponentially} in the shell number $k$.

In contrast to (\ref{time_locality}), since $\tilde{E}_0=R_0^2E_0$, $\tilde{P}_0=R_0^2P_0$,
$\tilde{\Psi}_{2R_2,R_2}=R_2^2{\Psi}_{2R_2,R_2}$ and $\tilde{\Psi}_R=R^2{\Psi}_R$,
\begin{equation}\label{space_time_locality}
\frac{c_{0,\gamma}}{c_{1,\gamma}}\le
\frac{{\Psi}_{2R_2,R_2}}{{\Psi}_R}\le\frac{c_{1,\gamma}}{c_{0,\gamma}}\;,
\end{equation}
i.e., the ensemble averages of the time-\emph{space} averaged
modified fluxes of the flows satisfying (\ref{scales_con_fin}) are
comparable throughout the scales involved in the inertial range
(\ref{inert_range}) which is consistent with the existence of the
enstrophy cascade.

We conclude this section by noticing that the remarks similar to those at the
end of section \ref{balls}
can be applied here. Namely we have the following.

\begin{obs}\label{E'_rem3}{\em
If the additional assumptions (\ref{psi_shells_def_add1}) and (\ref{psi_shells_def_add2})
for the test functions on the shells $A(\bfx_i,2R,R)$ which are not contained
entirely in $B(\bfo,R_0)$ are not imposed, then the lower bounds in
(\ref{Psi_R12_low}) and (\ref{ener_casc2}) hold with $\tilde{P}_0$ replaced
by the time average of  the {\em{non-localized}} in space enstrophy on $B(\bfo, R_0)$,
\[
\tilde{P}'_0=\frac{1}{T}\int\limits_0^{2T}
\int\limits_{B(\bfx_\bfo,R_0)}|\nabla\otimes\omega|^2\eta\,d\bfx\,dt=R_0^2P'\;.
\]
This is the case because the estimate (\ref{E_R12_E_ineq}) gets replaced with
\[\tilde{P}_{2R,R}\ge\frac{1}{K_1}\left(\frac{R}{R_0}\right)^2\tilde{P}'_0\;.\]

Also, the estimates (\ref{time_locality}) and
(\ref{space_time_locality}) will contain the terms
$P'_0/P_0(=\tilde{P}'_0/\tilde{P}_0)$ in the lower and $P_0/P'_0$ in the upper
bounds. }\end{obs}

\begin{obs}\label{unif_ave-obs2} {\em
If we integrate the relation (\ref{loc_enst_eq}) over $B({\bf 0}, R_0)$ (instead of
summing over the optimal covering) and use Lemma \ref{U_ave_lem}, the $\tilde{\Psi}_{2R,R}$ in
Theorem \ref{shells_thm}
can be replaced with the uniform averaged enstrophy flux into shells of thickness $R$,
\[\tilde{\Psi}^{u}_{2R,R}=\frac{1}{R_0^2}\int\limits_{B({\bf 0}, R_0)} \Psi_{\bfx,2R,R}\; d\bfx\;,\]
with $K_1=2^2$ and $K_2=4^2$.}
\end{obs}

\begin{obs}
{\em Working with (\ref{loc_ene_ineq}) yields similar results for the locality of the energy
fluxes $\Phi_{\bfx_0,R_1,R_2}$ and $\tilde{\Phi}_{2R,R}$. Namely, Theorems
\ref{shells_thm1} and \ref{shells_thm} hold with $\Psi$ replaced with $\Phi$, $P$
replaced with $E'$ and length scales $\sigma$ in the sufficient conditions
(\ref{scales_con2_fin})
and (\ref{scales_con_fin}) replaced with $\tau$.
}
\end{obs}

The locality of energy flux into shells related to the inverse
energy cascades is established in similar way (except, because of
the no-slip boundary condition (\ref{no_slip}) we can set
$\psi_0\equiv 1$). Note that the flux on a shell is defined exactly
in the same way as in (\ref{shell_ene_flux_def}), and we obtain the
exact equivalent of Theorem \ref{shells_thm1} in this setting. If in
addition the sufficient condition for inverse energy cascade
(\ref{back_ene_casc_con}) holds, then for $R_0<R_2<R_1<D/2$ we can
prove the following equivalent of Theorem \ref{shells_thm}.

\begin{thm}\label{back_shells_thm}
Assume that the condition (\ref{back_ene_casc_con}) holds for some
$0<\gamma<1/\sqrt{2C_0}$. Then, for any $R$ satisfying
(\ref{back_iner_range}), the ensemble average of the time-averaged
total energy flux out of the shells of thickness $R$,
$\tilde{\bar{\Phi}}_{2R,R}$,  satisfies
\begin{equation}\label{back_ener_casc2}
\bar{c}_{0,\gamma}\left(\frac{R}{R_0}\right)^2\,\nu{E}\le\tilde{\bar{\Phi}}_{2R,R}\le
\bar{c}_{1,\gamma} \left(\frac{R}{R_0}\right)^2\,\nu{E}\;,
\end{equation}
where $E$ is as in (\ref{back_E_def}), $\bar{c}_{0,\gamma}$ and
$\bar{c}_{1,\gamma}$ are defined in (\ref{back_casc_constants}), and
the average is computed over a time interval $[0,T]$ with $T\ge
R_0^2/\nu$ and determined by an optimal covering
$\{A(\bfx_i,2R,R)\}_{i=1,n}$ of $B(\bfo,R_0)$ (i.e. satisfying
(\ref{x_i_shells_con}), (\ref{shells_n_con1}), and
(\ref{Shells_n_con2})).
\end{thm}

\bibliographystyle{plain}
\bibliography{DG2}

\end{document}